\documentclass[a4paper,10pt]{amsart}

\usepackage{amsmath,amsfonts,amssymb,amsthm,mathrsfs}
\usepackage{hyperref}
\usepackage{bm}
\usepackage{color}
\usepackage{esint}
\usepackage{graphicx}
\usepackage{graphics}
\usepackage{geometry}
\usepackage[all]{xy}

\usepackage{tikz}

\usepackage{hyperref}

\renewcommand{\epsilon}{\varepsilon}

\newcommand{\N}{{\mathbb N}}

\newcommand{\C}{{\mathbb C}}

\renewcommand{\phi}{\varphi}

\newcommand{\bcal}{\mathcal{B}}

\newcommand{\hcal}{\mathcal{H}}

\newtheorem{theorem}{{Theorem}}[section]

\newtheorem{lem}[theorem]{{Lemma}}
\newtheorem{prop}[theorem]{{Proposition}}

\newenvironment{remark}{\medskip\noindent{\it Remark:\/} }{\medskip}

\theoremstyle{definition}

\numberwithin{equation}{section}

\def \N {\mathbb N}

\def \C {\mathbb C}

\def \blambda {\bm{\lambda}}
\def \barlambda {\bar{\bm\lambda}}

\def \P {\mathbb P}

\title[An infinite dimensional balanced embedding problem]{ An infinite dimensional balanced embedding problem I:existence}

\author{Jingzhou Sun and Song Sun}
\thanks{The first author is partially supported by NNSF of China no.11701353. The second author is partially supported by an NSF grant and the Simons Collaboration Grant in Special Holonomy.  }
\address{Department of Mathematics, Shantou University, Shantou City, Guangdong Province 515063, China}
\address{Department of Mathematics, University of California-Berkeley, Berkeley, California 94720, USA}
\email{jzsun@stu.edu.cn, sosun@berkeley.edu}

\begin{document}
	\begin{abstract}
		We investigate the problem of balanced embedding of a non-compact complex manifold into an infinite-dimensional projective space. In this paper we prove the existence  of such an embedding in a model case. The strategy is by using a gradient flow in a Hilbert space; both the long-time existence and convergence at infinite time are non-trivial. The long time existence is established by choosing a perturbation of the ODE; the convergence depends on a priori bounds that uses techniques in the proof of the Tauber-Hardy-Littlewood theorem.
	\end{abstract}
	
	\maketitle
	

	\section{Introduction}
	
	In this paper we study a model problem concerning balanced embedding of a non-compact complex manifold into a projective Hilbert space. 
	
	The general set-up is as follows: let $\mathcal H$ be a separable complex Hilbert space, and let $\mathbb P(\mathcal H)$ be the associated projective Hilbert space, viewed as a set. We denote by $\mathfrak h(\mathcal H)$ the space of bounded self-adjoint operators on $\mathcal H$. There exists a natural injective map 
	$$\iota: \mathbb P(\mathcal H)\rightarrow \mathfrak h(\mathcal H); [z]\mapsto \frac{z\otimes z^*}{|z|^2},$$
	where $z\otimes z^*$ is the operator defined by $y\mapsto \langle y,z\rangle z,y\in \hcal$.
	Let $X$ be a finite-dimensional (possibly non-compact) complex manifold with a volume form $d\mu$. We say a map $F: X\rightarrow \mathbb P(\mathcal H)$ is a holomorphic embedding if it is injective and locally induced by a holomorphic embedding into $\mathcal H$.
	
	Often $F$ arises from an \emph{$L^2$ embedding}. More specifically, suppose $(L, h)$ is a hermitian holomorphic line bundle over $X$. Then, we can define $\mathcal H$ as the dual of the Hilbert space of $L^2$ integrable holomorphic sections of $L$. The map $F$ can be given by the natural evaluation map.
	We say a holomorphic embedding $F: X\rightarrow \mathcal H$ is  \emph{balanced} if it satisfies the equation
	\begin{equation}
	c(F)\equiv\int_{X} \iota(F(z)) d\mu=C\cdot\text{Id}, 
	\end{equation}
	where $C$ is a constant.
	
	Such balanced embeddings into projective spaces are well-studied when $\mathcal H$ is finite-dimensional. See for example \cite{BLY}. There is a more non-linear variant of the problem when the involved volume form is induced from the Fubini-Study metric on the projective space $\mathbb P(\mathcal H)$. 
	The latter has connections with geometric invariant theory, which is particularly relevant for the finite-dimensional quantization of geometric partial differential equations on projective manifolds (see for example Donaldson \cite{donaldson2001}).

	More generally, we consider a pair $(X, D)$, where $D$ is a divisor in $X$. Let $d\mu$ be a volume form on $X$ and $d\nu$ be a volume form on $D$. We say that a holomorphic embedding $F: X\rightarrow \mathbb P(\mathcal H)$ is $(\beta, C)$-balanced for some $\beta, C\in \mathbb R$ if
	\begin{equation}\label{first balance equation} \int_X \iota(F(z)) d\mu+\beta \int_D \iota(F(z)) d\nu=C\cdot\text{Id}. 	
	\end{equation}
 The motivation for this concept comes from its potential connection with GIT stability and the existence of canonical metrics on $X$ with singularities along $D$ (see \cite{Donaldson10, Sun, SS} for discussion). Note that in finite dimensions, by taking the trace of Equation \eqref{first balance equation} we see the constant $C$ is determined by $\beta$ , but in infinite dimensions this is not the case.
	
	In this paper, our focus is on  a local model case. Specifically, we consider the setting where $X$ is the complex plane $\mathbb C$ equipped with the standard volume form $d\mu =\frac{\omega_0}{\pi}$, and $D$ is the origin ${0}$. We take $L$ to be the trivial holomorphic line bundle. For any $\beta\in \mathbb R$, $L$ admits a singular hermitian metric $h_\beta \equiv |z|^{-2\beta}$ whose Chern connection is flat with holonomy $2\pi\beta$ around $D$. This singular hermitian metric can be viewed as a canonical metric associated with the triple $(L, D, \beta)$. Our main result in this paper proves an algebraic counterpart. Below, $\mathcal H$ is the Hilbert space of entire holomorphic functions on $\mathbb C$ endowed with the $L^2$ inner product associated with $h_0$.

	\begin{theorem}\label{t:main}
		For $\beta\in [0, 1)$ there exists a $(\beta ,1)$-balanced embedding $F: X\rightarrow \mathbb P(\mathcal H)$ . 
	\end{theorem}

	\begin{remark}
		\begin{itemize}
		\item By analogy with the Kempf-Ness theorem in finite dimensions, one may formally view this result as saying that the embedding of $(X, D)$ in $\mathbb P(\mathcal H)$ should  be ``Chow polystable", but the latter notion in infinite dimensions has not yet been developed in the literature. 
			\item  The uniqueness and geometry of such embedding, which is  highly non-trivial and much more involved, will be  studied in a sequel to this article.
			\item The case when $\beta=0$ admits an easy explicit solution as we explain below;  the case when $\beta=1$ admits an explicit degenerate solution.
		\end{itemize}	
	\end{remark}
	
	We will now give an outline of the main steps in the proof. By using the action of $U(\mathcal H)$, we can assume that a holomorphic embedding $F: X\rightarrow \mathbb P(\mathcal H)$ has the form $F(z)=[a_0: a_1z^1:a_2z^2: \cdots]$. We denote $e^{\lambda_i}=|a_i|^2$ and define the associated function on $[0, \infty)$ as follows:
	$$f(x)=\sum_{i=0}^\infty e^{\lambda_i}x^i,$$
	The condition of being $(\beta, 1)$-balanced becomes the system of equations satisfied by $\blambda=\{\lambda_i\}_{i=0}^\infty$.
	\begin{equation} \label{e-1}
	F_i(\blambda)\equiv \int_{0}^{\infty}\frac{e^{\lambda_i} x^i}{\sum_{j=0}^{\infty}e^{\lambda_j}x^j}dx-(1-\beta\delta_{0i})=0, \ \   i=0, 1, \cdots
	\end{equation}
	This is equivalent to the condition that 
	\begin{equation}\label{e-2}
	\int_0^{\infty} \frac{f(sx)}{f(x)}dx=\frac{1}{1-s}-\beta, 
	\end{equation}
	for all $s\in[0,1)$. 
	Theorem \ref{t:main} then reduces to the statement that there is a unique (up to constant multiple) solution $\blambda$ to \eqref{e-1}. 
	
	When $\beta=0$ there is an explicit solution $\barlambda$ to \eqref{e-1} given by $\bar\lambda_i=\log\frac{1}{\Gamma(i+1)}$. 
	For $\beta\in (0,1)$ we do not have explicit solutions.  To prove the existence of a solution we consider the natural flow equation 
	\begin{equation}\frac{d}{dt}\lambda_i(t)=F_i(\blambda(t)), \ \ \ i=0, \cdots, 
	\label{e:ODE first}	
	\end{equation}
	with initial data $\blambda(0)=\barlambda$. 
	This is a non-linear ODE in infinite dimensions. Formally it is the gradient flow of a Kempf-Ness function and it decreases the energy function 
	$$E(\blambda)\equiv \sum_i(F_i(\blambda))^2. $$
	However, it does not seem easy to directly find a solution $\blambda(t)$ with finite energy, even for a short time. To overcome this problem, our idea is to introduce a family of perturbed ODEs for $s\in (0, 1)$
	$$\frac{d}{dt}\lambda_i(t)=s^iF_i(\blambda(t)), \ \ \ i=0, \cdots.$$
	We will first prove the existence of a solution $\blambda^s(t) (t\in [0, \infty))$ for all $s\in (0,1)$. Then we take limits as $s\rightarrow1$ to get a  solution $\blambda(t)$ to \eqref{e:ODE first} for $t\in [0, \infty)$. It has the property that $E(\blambda(t))$ is finite and decreasing in $t$. It also follows that we have $\|\blambda(t)-\barlambda\|_{\ell^2}\leq Ct$. 
	
	To obtain a solution to \eqref{e-1} it remains to take limits as $t\rightarrow\infty$. This involves a priori bounds on the components of $\blambda(t)$ under the finite energy assumption. However in general we do not expect a strong uniform bound such as $\|\blambda(t)-\barlambda\|_{\ell^\infty}$. By analogy, in the finite dimensional situation such a uniform bound only holds when we are in the GIT-stable situation. In our infinite dimensional setting we do not have a notion of GIT stability. Instead we can only prove uniform bound on $\|\lambda_i(t)-\bar\lambda_i\|$ for each fixed $i$, which suffices to give a solution to \eqref{e-1}, though we lose the asymptotic control of the solution. This is similar to the possible jumping of orbits in the GIT-unstable case in finite dimensions. Our argument uses the idea of Hall-Williamson \cite{HW} and is related to the Tauber-Hardy-Littlewood theorem in analysis.

One motivation for this work is related to the longstanding difficult question of asymptotic stability of singular polarized varieties with canonical metrics, which is an  extension of Donaldson's theorem in \cite{donaldson2001} for smooth manifolds. The idea of proof in \cite{donaldson2001} goes as follows, suppose there is a polarized K\"ahler manifold $(X, L)$ with $Aut(X, L)$ discrete and suppose $\omega$ is a K\"ahler metric with constant scalar curvature in $2\pi c_1(L)$, then for $k$ large, using the asymptotic expansion of the Bergman kernels, one can show that the induced $L^2$ embedding of $X$ into $\P(H^0(X, L^k)^*)$ is approximately balanced. A conceptual reason for this is that around each point $p$, as $k\rightarrow\infty$ locally the problem is modeled  by the trivial line bundle over $\C^n$ endowed with the Guassian hermitian metric $e^{-|z|^2/2}$, and the $L^2$ embedding of $\C^n$ into $\mathbb P(\mathcal H)$ is \emph{exactly} balanced.  Then one makes use of a quantitative perturbation theorem to make the embedding precisely balanced, which then implies asymptotic Chow stability. When $(X, \omega)$ admits mild singularities (for example, when $\omega$ has conical singularities along a divisor), the situation is much more complicated. A conceptual difficulty is that around the singular locus one needs to use a different model than the above, which is related to the notion of $(\beta, C)$-balanced embedding introduced  before. The problem we studied here is a simplified model. Our investigations demonstrate the difficulty of the such problem and we hope the analysis could be further extended and improved.  For example, for applications to the above problem, one would like to prove the uniqueness of solution to \eqref{e-2} and find a much more quantitative understanding. This will be discussed in the sequel.


	This paper is organized as follows. In Section 2 we  prove the long time existence of the solution of the evolution equation. In section 3 we prove a priori estimates for the evolution equation. In section 4 we prove Theorem \ref{t:main}.

	\subsection*{Acknowledgements}  The second author is grateful to Professor Simon Donaldson for suggestions and discussions when part of the results was derived in 2010-2011 at Imperial College London. We thank Yueqing Feng for reading the article and providing comments. 
	
	\section{The evolution equation and long time existence}

	For $p\in [1, \infty]$ we denote by ${\ell}^p$ the Banach space of all sequences $\blambda=\{\lambda_i\}_{i=0}^\infty$ endowed with the $l^p$ norm $\|\cdot\|_{\ell^p}$.  We fix a sequence $\bar{\bm \lambda}\in \ell^2$ given by
	\begin{equation*} e^{\bar{\lambda}_i}=\frac{1}{\Gamma(i+1)}.
	\end{equation*}
	Denote by $\tilde {\ell}^p$ the shifted space consisting of all sequences $\blambda$ such that $\blambda-\bar{\bm \lambda}\in \ell^p$. 
	We define two mappings
	$$\bm I, \bm F: \tilde \ell^\infty\rightarrow \ell^\infty$$
	via
	\begin{equation*}
	I_i(\blambda)=\int_{0}^{\infty} \frac{e^{\lambda_i}x^i}{\sum_{j=0}^{\infty}e^{\lambda_j}x^j}dx, 
	\end{equation*} 
	and
	\begin{equation*}  F_i(\blambda)=\int_{0}^{\infty} \frac{e^{\lambda_i}x^i}{\sum_{j=0}^{\infty}e^{\lambda_j}x^j}dx-1+\beta\delta_{0i}.
	\end{equation*} 
	Notice that $I_i(\barlambda)=1$ for all $i$. It is easy to see both $\bm I$ and $\bm F$ are well-defined. Indeed,  for $\bm\lambda,\bm\lambda'\in \tilde{\ell}_\infty$, one can see that
	\begin{equation}\label{e:I comparions}e^{-2\|\bm\lambda-\bm\lambda'\|_{\ell^\infty}}\leq \frac{I_i(\bm\lambda)}{I_i(\bm\lambda')}\leq e^{2\|\bm\lambda-\bm\lambda'\|_{\ell^\infty}}, 	
	\end{equation}
	and
	\begin{equation}{|F_i(\bm\lambda)-F_i(\bm\lambda')|}\leq (e^{2\|\bm\lambda-\bm\lambda'\|_{\ell^\infty}}-1){I_i(\bm\lambda)}.
	\label{e:F comparison}	
	\end{equation}

	We consider the ODE 
	\begin{equation}
	\label{e-evo}	
	\frac{d}{dt}\blambda(t)=-\bm F(\blambda(t)).
	\end{equation}
	We say $\blambda(t)$ is a $\tilde\ell^p$ solution for $t$ in an interval $J$ if $t\mapsto \blambda(t)$ is a $C^1$ map from $J$ into $\tilde \ell^p$ and the equation holds for all $t\in J$. 
	The main result of this subsection is
	\begin{theorem} \label{t: ODE existence}
		There is a unique $\tilde \ell^2$ solution $\blambda(t)$ to \eqref{e-evo} for $t\in [0, \infty)$ with $\blambda(0)=\barlambda$.
	\end{theorem}
	It turns out that even the short time existence does not follow directly from the general theory. We will instead consider first a family of perturbed equations. For fixed $s\in (0, 1]$, we define $$\bm F^s: \tilde\ell^\infty\rightarrow \ell^\infty$$ by setting $$F^s_i(\blambda)=s^i \cdot F_i(\blambda)$$ for all $i$. Notice  that if $s<1$, then $\bm F^s$ has image in $\ell^2$--this is a crucial fact for us. We denote $$G^s(\bm\lambda)\equiv \sup_i s^iI_i(\bm\lambda)$$ and $$H^s(\bm\lambda)\equiv \|\bm F^s(\bm\lambda)\|_{\ell^\infty}.$$

	We consider the perturbed equations
	\begin{equation}
	\label{e-evos}	
	\frac{d}{dt}\blambda(t)=-\bm F^s(\bm\lambda(t))
	\tag{*s}
	\end{equation}
	
	\begin{prop} \label{p:perturbed ODE}
		For each $s\in (0, 1)$, there is a unique solution $\blambda^s(t)\in \tilde\ell^2$ to \eqref{e-evos} for $t\in [0,\infty)$ with $\blambda^s(0)=\barlambda$.
	\end{prop}

	To prove this we use the classical Picard-Lindel\"of theory. Given $b>0$ we  consider the space $\bcal\equiv C([0, b],\ell^\infty)$ of continuous maps from $[0, b]$ to $\ell^\infty$. Equipped with the sup-norm, $\bcal$ is then a Banach space. Denote by $D_r(\xi)$ the closed ball of radius $r$ centered at $\xi$ in $\bcal$.
	
	Suppose we are given $\bm\lambda_0=\barlambda+\bm\mu_0\in \tilde\ell^\infty$. 
	Writing $\bm\lambda(t)=\barlambda+\bm\mu(t)$, then $\bm\lambda(t)(t\in [0, b])$ is a solution to \eqref{e-evos} in $\tilde\ell^\infty$ with $\bm\lambda(0)=\bm\lambda_0$  if and only if $\bm\mu(t)$ is a fixed point of the integral operator $T:\bcal\to \bcal$ defined by
	$$T(\bm\mu(t))(x)=\bm\mu_0+\int_{0}^{x}\bm F^s(\bm\mu(t)+\barlambda)dt. $$
	By \eqref{e:I comparions} and \eqref{e:F comparison} we have 
	$$\| \bm F^s(\bm\mu+\barlambda)-\bm F^s(\bm\mu'+\barlambda)\|_{\ell^\infty}
	\leq21 G^s(\bm\lambda_0)\|\bm\mu-\bm\mu'\|_{\ell^\infty}$$ 
	whenever $\|\bm\mu-\bm\mu_0\|\leq \frac12$ and $\|\bm\mu'-\bm\mu_0\|\leq \frac12$.
	
	Denote by $\tilde{\bm\mu}_0\in \bcal$ the constant map $\tilde{\bm\mu}_0(t)\equiv \bm\mu_0$.  Then we have 
	$$\| T(\tilde{\bm\mu}_0)-\tilde{\bm\mu}_0 \|_{\bcal} \leq bH^s(\blambda_0), $$ and
	$$\| T(\bm\mu(t))-T(\bm\mu'(t)) \|_{\bcal}\leq 21bG^s(\bm\lambda_0)\| \bm\mu(t)-\bm\mu'(t) \|_{\ell^\infty} $$
	for $ \bm\mu(t),\bm\mu'(t) \in  D_{\frac12}(\tilde{\bm\mu}_0)$. 
	
	Let $b=({3H^s(\blambda_0)+100G^s(\blambda_0)})^{-1}$, then $T$ is a contraction mapping from $D_{\frac12}(\tilde{\bm\mu}_0)$ into itself. By the Banach fixed point theorem we get 
	
	\begin{lem}\label{l:short time existence}
		For any $s\in (0, 1]$, there exists a unique solution $\tilde\ell^\infty$ solution $\blambda^s(t)$ to the equation \eqref{e-evos} for $t\in [0, b]$ with $\blambda^s(0)=\blambda_0$. \end{lem}
	In particular, by  setting $\blambda_0=\barlambda$, we obtain a unique $\tilde\ell^\infty$  solution $\blambda^s(t)$ to the equation \eqref{e-evos} with $\blambda^s(0)=\barlambda$ for all $s\in (0, 1]$ and $t\in [0, \underline b]$ for some $\underline b>0$.
	
	\begin{proof}[Proof of Proposition \ref{p:perturbed ODE}]
		Fix $s\in (0,1)$, to prove the long time existence to \eqref{e-evos} it suffices to obtain a uniform  a priori bound of $G^s(\blambda^s(t))$ and $H^s(\blambda^s(t))$. To simplify notation we will drop the superscript $s$ in $\blambda^s(t)$.  We define the perturbed energy function on $\tilde\ell^\infty$ by setting
		\begin{equation}
		E_s(\blambda)\equiv\sum_i s^i\cdot (F_i(\blambda))^2.
		\end{equation}
		Suppose the flow \eqref{e-evos} exists in $\ell^\infty$ for $t\in[0, T)$ for some $T<\infty$. Then we have $\bm F(\blambda(t))\in \ell^{\infty}$.  In particular,  we have $E_s(\bm\lambda(t))<\infty$. Using the equation it follows that $\blambda(t)\in \tilde\ell^2$ for $t<T$. Moreover, we have 
		\begin{eqnarray*}
			\frac{d}{dt}E_s(\blambda(t))&=&2\sum_{i=0}^{\infty}\int_0^\infty[\frac{\dot{\lambda}_ie^{\lambda_i}x^i\sum_{j=0}^{\infty}\dot{\lambda}_je^{\lambda_j} x^j}{(\sum_{j=0}^{\infty}e^{\lambda_j} x^j)^2}-\frac{\dot{\lambda}^2_ie^{\lambda_i}x^i}{\sum_{j=0}^{\infty}e^{\lambda_j} x^j}]dx\\
			&=&-\sum_{i,j}	\int_0^\infty \frac{e^{\lambda_i+\lambda_j}x^{i+j}(\dot{\lambda}_i-\dot{\lambda}_j )^2  }{(\sum_{k=0}^{\infty}e^{\lambda_k} x^k)^2}dx
		\end{eqnarray*}
		Since $E_s(\blambda(0))=\beta^2$, we get that $E_s(\blambda(t))\leq \beta^2$ for $t\in [0, T)$. Immediately we obtain 
		$$H^s(\blambda(t))^2\leq \sum_i (F^s_i(\blambda(t)))^2\leq E_s(\blambda(t))\leq \beta^2\leq1.$$
		Moreover, for each $i$ we have 
		$$|F_i(\blambda(t))|\leq \beta{s^{-i/2}}.$$
		So  we see $I_i(\blambda(t))<1+{\beta}{s^{-i/2}}$. Hence $$G^s(\blambda(t))\leq \sup_i s^i(1+\beta s^{-i/2})\leq 2$$ is  also uniformly bounded. So we can apply Lemma \ref{l:short time existence} to get a $\tilde\ell^\infty$ solution $\blambda^s(t)$ for $t\in [0, \infty)$. The above discussion also shows that $\blambda(t) $ is indeed a $\tilde \ell^2$ solution.
		
	\end{proof}
	
	\begin{proof}[Proof of Theorem \ref{t: ODE existence}]
		
		We first notice from the above proof of Proposition \ref{p:perturbed ODE} that for all $s\in (0, 1)$ and $t\in [0, \infty)$ we have  $\blambda^s(t)\in \tilde{\ell}^2$. Moreover, we have  
		\begin{equation}
		\| \blambda^s(t)-\bar{\blambda}\|_{\ell^2}\leq \int_0^t\|F^s(\blambda^s(u))\|_{\ell^2}du\leq \beta t.
		\end{equation}
		Now we fix a number $A>0$. For $t\in[0,A]$, for each $i$ we have 
		$$|\lambda^s_i(t)|\leq \bar{\lambda}_i+\beta A,$$
		and 
		$$|\frac{d}{dt}{\lambda}^s_i(t)| =|F^s_i(\blambda^s(t))|\leq \beta.$$
		Therefore, we can apply the theorem of Arzela-Ascoli to get a sequence $s_k\to 1$ such that 
		$\lambda_i^{s_k}(t)$ converges uniformly to a function $\lambda_i(t)$ on $[0,A]$. Then by a diagonal argument, we get a subsequence, which, by abuse of notation, we still denote by $s_k\to 1$, such that 
		$\lambda_i^{s_k}(t)$ converges uniformly to a function $\lambda_i(t)$ on $[0,A]$ for all $i$. It remains to show that $\blambda(t)=\{\lambda_i(t)\}$ satisfies the desired equation.

		Recall that for $s\in (0, 1)$ we have $$\lambda^s_i(x)=\bar{\lambda}_i+\int_0^x F^s_i(\blambda^s(t))dt.$$
		For fixed $i$, as $s_k\to 1$, the left hand side converges uniformly to $\lambda_i(x)$, whereas the right hand side converges to 
		$\bar{\lambda}_i+\int_0^x F^s_i(\blambda(t))dt$ by the lemma below and the dominated convergence. 
		Therefore, $\blambda(t)$ satisfies the evolution equation \eqref{e-evo} for $t\in [0, A]$. Notice that by our convergence for each $t\in [0,A]$,
		we have the uniform estimates \begin{equation}\label{e-F}
		\| F(\blambda(t))\|_{l^2}\leq \beta,
		\end{equation}
		and
		\begin{equation}\label{e-lambda}
		\| \blambda(t)\|_{\tilde{\ell}^2}\leq\beta t.
		\end{equation}
		Now we can let $A\rightarrow\infty$, a diagonal sequence argument yields a $\tilde \ell^2$ solution $\bm\lambda(t)$  of \eqref{e-evo} for all $t\in [0, \infty)$.  The uniqueness of solution follows from Lemma \ref{l:short time existence}.
	\end{proof}

	Given $\blambda\in \tilde{\ell}^2$, we associate to it the function $f_{\blambda}(x)=\sum_{i=0}^{\infty}e^{\lambda_i}x^i$ for $x\in [0,\infty)$.
	
	\begin{lem}\label{lem-e}
		
		Let $\blambda^j\in \tilde{\ell}^2$ be sequence such that $\| \blambda^j\|_{\tilde{\ell}^2}\leq C$ for a fixed $C>0$. Suppose $\blambda^\infty\in \tilde{\ell}^2$ satisfies  that 
		$$\lim_{j\to \infty}\lambda^j_i=\lambda^\infty_i$$
		for each $i$. Then $$\lim_{j\to \infty}f_{\blambda^j}(x)=f_{\blambda^\infty}(x),$$
		for all $x\in [0,\infty)$. Consequently,
		$$\lim_{j\to \infty}I_i(\blambda^j)=I_i(\blambda^\infty)$$
		for each $i$.
	\end{lem}
	\begin{proof}
		Fix $x\in [0,\infty)$, then for any $\epsilon>0$, there exists $N>0$ such that $\sum_{i=N}^{\infty}e^{\bar{\lambda}_i}x^i<\epsilon$. Therefore, 
		$$\sum_{i=N}^{\infty}e^{\lambda^j_i}x^i<e^C\epsilon,$$
		for all $j$. But $\sum_{i=0}^{N+1}e^{\lambda^j_i}x^i$ converges to $\sum_{i=0}^{N+1}e^{\lambda^\infty_i}x^i$ by the assumption. Therefore,  $$\lim_{j\to \infty}f_{\lambda^j}(x)=f_{\lambda^\infty}(x).$$
		
		For the second part, the argument is similar. We only need to notice that if we fix $i$, then for any $\epsilon>0$, there exists $N>0$ such that $$\int_{N}^{\infty}\frac{e^{\bar{\lambda}_i}x^i}{e^x}dx<\epsilon.$$
		Then by dominated convergence, we get the conclusion.
	\end{proof}

	\
	
	We define the (unperturbed) energy function on $\tilde\ell^2$ by
	$$E(\blambda)=\sum_{i=0}^{\infty}F_i^2(\blambda).$$
	Since $E_s(\blambda^s(t)))\leq \beta^2$, by Fatou's Lemma $E(\blambda(t))\leq \beta^2$.
	Then along the solution $\blambda(t)$, we have
	\begin{equation}
	\frac{d}{dt}E(\blambda(t))=-\sum_{i,j}	\int_0^\infty \frac{e^{\lambda_i+\lambda_j}x^{i+j}(\dot{\lambda}_i-\dot{\lambda}_j )^2  }{(\sum_{k=0}^{\infty}e^{\lambda_k} x^k)^2}dx<0
	\end{equation}
	
	\section{A priori bounds}
	
	In order to study the limits of $\blambda(t)$ as $t\to \infty$, we need to obtain a priori bounds. For our purpose we only need to take limits of $\blambda(t)$ modulo a constant sequence so we can always normalize so that $\lambda_0(t)=0$. Notice that after normalization $\blambda(t)$ may not be in $\tilde l^2$ but still in $\tilde l^\infty$.

	Now we consider a sequence $\blambda\in \tilde l^\infty$, normalized so that $\lambda_0=0$. We assume $E(\blambda)\leq\beta^2$, which holds along the above flow $\blambda(t)$. As before we define the function
	$f(x)=\sum_{i=0}^{\infty} e^{\lambda_i} x^i$, $x\in [0,\infty)$. The normalization implies that  $f(0)=1$. In the following we will derive uniform bounds on $\lambda_i$ depending only on $\beta$. The lower bound has no restriction on $\beta$, but the upper bound will require $\beta<\frac{1}{2}$. 
	
	\subsubsection{Lower bounds}\label{subset-low}
	We use an observation of Hall-Williamson \cite{HW}. 
	Notice  that $$u(x)=xf'(x)/f(x)$$ is strictly increasing in $x$, so for any $x>0$ and $t\in (0,1)$ we have
	$$-\frac{t}{1-t}x \frac{f'( x)}{f(x)} \leq \log \frac{f((1-t)x)}{f(x)}\leq  -t(1-t)x\frac{f'((1-t)x)}{f((1-t)x)}.$$
	From  this one easily obtains
	$$(1-t) \int_0^{\infty} \frac{f((1-t)x)}{f(x)}dx\leq \int_0^{\infty}e^{-tx\frac{f'(x)}{f(x)}}dx \leq  \int_0^{\infty} \frac{f(x/(1+t))}{f(x)}dx.$$
	By the energy  bound assumption we have
	$$\int_0^{\infty}\frac{f((1-t)x)}{f(x)}dx=\sum_{i=0}^{\infty}[\int_0^{\infty} \frac{e^{\lambda_i} x^i}{\sum_{j=0}^{\infty}e^{\lambda _j}x^j}dx](1-t)^i=\frac{1}{t}-\beta+\sum_{i=0}^{\infty}F_i(\blambda)(1-t)^i, $$
	and 
	$$\int_0^{\infty}\frac{f(x/(1+t))}{f(x)}dx=\sum_{i=0}^{\infty}[\int_0^{\infty} \frac{e^{\lambda_i} x^i}{\sum_{j=0}^{\infty}e^{\lambda _j}x^j}dx](1+t)^{-i}=\frac{1+t}{t}-\beta+\sum_{i=0}^{\infty}F_i(\blambda)t^i(1+t)^{-i}. $$
	Noticing that   for all $t\in (0,1)$
	$$\sum_{i=0}^{\infty} F_i(\blambda)(1-t)^i\leq \beta\sqrt{\sum_{i=0}^{\infty} (1-t)^{2i}}\leq \beta t^{-\frac 1 2}. $$
	Now view $u$ as a variable and denote $\phi(u)=\frac{dx}{du}$, then we get  for all $t>0$ that 
	$$|\int_0^{\infty} e^{-tu} \phi(u) du-\frac{1}{t}|\leq C_0\cdot \frac{t^{1/2}+1}{t^{1/2}}, $$
	where the constant $C_0$ depends only on $\beta$.
	Now we use the idea of Karamata in the proof of Tauber-Hardy-Littlewood theorem. Let $g$ be any measurable function on $[0,1]$ with bounded variation. Fix any $\epsilon>0$ small, by the Stone-Weierstrass approximation theorem there are polynomials $p$ and $P$ such that
	$$p(x)< g(x)< P(x), $$ 
	and
	$$\int_0^{\infty} e^{-u}(P(e^{-u})-p(e^{-u}))dt< \epsilon.$$
	For any $n\in \N$, we have
	$$|\int_0^{\infty} e^{-tu}e^{-ntu} \phi(u) du-\frac{1}{(n+1)t}|\leq C_0\frac{1+(n+1)^{1/2}t^{1/2}}{(n+1)^{1/2}t^{1/2}}.$$
	Let $N$ be the maximum of $\deg p$ and $\deg P$. Then 
	$$|\int_0^{\infty} e^{-tu} P(e^{-tu})\phi(u)du-\frac{1}{t} \int_0^{\infty} e^{-u} P(e^{-u}) du|\leq C_1\sum_{n=0}^N\frac{1+(n+1)^{1/2}t^{1/2}}{(n+1)^{1/2}t^{1/2}}. $$
	So 
	$$\int_0^{\infty} e^{-tu} g(e^{-tu})\phi(u) du\leq \frac{1}{t}(\int_0^{\infty} e^{-u} g(e^{-u})du+\epsilon)+C_1 \sum_{n=0}^N\frac{1+(n+1)^{1/2}t^{1/2}}{(n+1)^{1/2}t^{1/2}}. $$
	Similarly we have
	$$\int_0^{\infty} e^{-tu} g(e^{-tu})\phi(t) dt\geq \frac{1}{t}(\int_0^{\infty} e^{-u} g(e^{-u})du-\epsilon)-C_1 \sum_{n=0}^N\frac{1+(n+1)^{1/2}t^{1/2}}{(n+1)^{1/2}t^{1/2}}. $$
	Here the constant $C_1$ depends only on $C_0$ and $g$.
	Now define $g(x)$ to be zero for $x< e^{-1}$ and $x^{-1}$ for $x\in [e^{-1}, 1]$. Then 
	
	$$\int_0^{\infty} e^{-u} g(e^{-u}) du =1, $$
	and
	
	$$\int_0^{\infty} e^{-tu} g(e^{-tu})\phi(u)du=\int_0^{t^{-1}} \phi(u) du. $$
	Let $T=t^{-1}$ and fix $\epsilon$ sufficiently small we obtain 
	$$\int_0^T \phi(u)du\leq (1+\epsilon)T+C _1\sum_{n=0}^N\frac{T^{1/2}+(n+1)^{1/2}}{(n+1)^{1/2}}. $$
	In particular 
	\begin{equation}  \label{e-6}
	\int_0^T \phi(u) du \leq C_2(T+1), 
	\end{equation} 
	for a different constant $C_2$, and for all $T>0$. 
	Similarly we have
	\begin{equation}\label{e-7}
	\int_0^T \phi(u)du\geq (1-\epsilon)T-C_2 \sum_{n=0}^N\frac{T^{1/2}+(n+1)^{1/2}}{(n+1)^{1/2}}. 
	\end{equation}
	Using (\ref{e-6}) we obtain that 
	\begin{equation} \label{e-8}
	x(u)\leq C_2(u+1),
	\end{equation}
	and then we have 
	$$\frac{f'(x)}{f(x)} = \frac{u}{x} \geq \frac u {C_2(u+1)}. $$ 
	So $$\log f(x)\geq \frac{1}{C_2} [u-\log (u+1)]$$
	From equation (\ref{e-7}) there is a constant $C_3>0$ such that 
	\begin{equation} \label{e-9}
	u\leq \max(C_3, 2x).
	\end{equation}
	So we have
	$$\log f(x)\geq \frac{1}{C_2}[\frac{x}{C_2}-1-\log \max(C_3, 2x)]\geq C_4x-C_5.$$
	Therefore we obtain 
	\begin{equation} \label{e-10}
	f(x)\geq e^{C_4x-C_5}.
	\end{equation}
	Using the assumption that $|F_i(\blambda)|^2\leq E(\blambda)\leq \beta^2$ for $i$ and the definition of $F_i$ we can easily derive a lower bound 
	\begin{equation} \label{e-11}
	e^{\lambda_i}\geq \frac{C_6C_5^{i+1}}{\Gamma(i+1)}
	\end{equation}
	for some constant $C_6>0$.
	
	\
	
	\subsubsection{Upper bounds}
	Now we turn to obtain upper bounds on $\blambda=\{\lambda_i\}$.
	Since
	\begin{eqnarray*}
		1-\beta+F_0(\blambda)
		&=&  \int_0^{\infty}\frac{e^{\lambda_0}}{\sum_{j=0}^{\infty}e^{\lambda_j} x^j}dx\leq \int_0^{\infty} \frac{1}{1+e^{\lambda_2}x^2}dx\\
		&=&\int_0^{\infty}\frac{1}{1+e^{\lambda_2}x^2}dx=e^{-\frac{\lambda_2}{2}} \frac{\pi}{2}, 
	\end{eqnarray*}
	we have
	\begin{equation}e^{\frac{\lambda_2}{2}}\leq \frac{\pi}{2(1-\beta+F_0(\blambda))}.
	\label{e:estimate on lambda_2}	
	\end{equation}
	Similarly, for all $i>0$
	$$e^{\frac{\lambda_{i+2}-\lambda_i}{2}} \leq \frac{\pi}{2(1+F_i(\blambda))}\leq \frac{\pi}{2(1-\beta)}, $$
	noticing that $|F_i(\blambda)|^2\leq E(\blambda)\leq\beta^2$ for $i>0$.
	Therefore for $n$ even,
	$$e^{\lambda_n}\leq C_7^ne^{\lambda_2},$$
	and for $n$ odd,
	$$e^{\lambda_n}\leq C_8^ne^{\lambda_1}.$$
	We have
	\begin{prop}
		For any $\delta>0$  there exists $C_\delta>0$ depending only on $\delta$ such that if $\beta\leq \frac12-\delta$, then $e^{\lambda_1}+e^{\lambda_2}\leq C_\delta$.
	\end{prop}
	\begin{proof}
		The estimate on $e^{\lambda_2}$ follows from \eqref{e:estimate on lambda_2} and the fact that $|F_0(\blambda)|^2\leq E(\blambda)\leq\beta^2$.
		For $e^{\lambda_1}$ we compute
		\begin{eqnarray*}
			1-\beta+F_0(\blambda)&=& \int_0^{\infty} \frac{1}{\sum_{j=0}^{\infty} e^{\lambda_j} x^j}dx\\
			&\leq & \int_0^{\infty} \frac{1}{1+e^{\lambda_1}x+Cx^2} dx,
		\end{eqnarray*}
		where we have used the previous lower bound on $e^{\lambda_2}$. 
		It is then easy to deduce the desired bound on $e^{\lambda_1}$.
	\end{proof}

	Thus we have achieved an upper bound that for $\beta\leq\frac12-\delta$
	\begin{equation}\label{e-12}
	e^{\lambda_i}\leq C_9^i 
	\end{equation}
	for all $i$.  From this and estimate \eqref{e-11}, it follows that the entire function $f(x)=\sum_{i=0}^{\infty}e^{\lambda_i} x^i$ satisfies 
	$$\frac{f'(x)}{f(x)}\leq C, $$ for a uniform constant $C$ and all $x\leq \frac{1}{2C_9}$.  Combining this with equation (\ref{e-9}) we see that 
	$f'(x)/f(x)$ is uniformly bounded for all $x$. Then we easily obtain a bound
	$$f(x)\leq e^{Cx}, $$
	for a different constant $C$,
	and so $$e^{-\lambda_i}\geq \frac{1}{1-\beta}\int_{0}^{\infty}\frac{x^i}{e^{Cx}}dx.$$
	Therefore
	\begin{equation}\label{e-13}
	e^{\lambda_i}\leq \frac{1}{C^{i+1}\Gamma(i+1)}, 
	\end{equation}  
	for some $C$ independent of $i$.

	\section{Infinite time convergence}
	
	We go back to the solution $\blambda(t)$ of \eqref{e-evo}. 
	We first consider the case $\beta<\frac12$.
	Since $E(\blambda(t))$ is non-negative, there is a sequence $t_q\to \infty$ such that $\frac{d}{dt}E(\blambda(t_q))\to 0$. By our estimates \eqref{e-11} and \eqref{e-13}, we must have $\dot{\lambda}_i(t_q)-\dot{\lambda}_j(t_q)\to 0$ for all pairs $i,j$.
	Since $\sum_{i=0}^\infty \dot{\lambda}_i(t_q)^2=E(\blambda(t_q))\leq E(\blambda(0))=\beta^2  $, we must have $\lim_{q\to\infty}\dot{\lambda}_i(t_q)=0$ for all $i$. Therefore, we have 
	$\lim_{q\to\infty}F_i(\blambda(t_q))=0$ for all $i$.
	
	We define the normalized sequence $\blambda'(t_q)$ by $\lambda'_i(t_{q})\equiv \lambda_i(t_{q})-\lambda_0(t_{q})$. Notice that $\bm F(\blambda'(t_q))=\bm(\blambda(t_q))$ . Then by the estimates \eqref{e-11} and \eqref{e-13}, for each $i$, we can choose a subsequence of $\{t_q\}$ so that $\lambda_i(\infty)=\lim_{k\to \infty}\lambda'_i(t_{q_k})$ exists. By a diagonal argument, we can find a subsequence which, by abuse of notation, we still denote by $\{t_q\}$ such that 
	$\lambda_i(\infty)=\lim_{q\to\infty}\lambda'_i(t_{q})$ exists for each $i$. Then by \eqref{e-11} and \eqref{e-13} again, we have that $f_\infty(x)=\sum_{i=0}^\infty e^{\lambda_i(\infty)}x^i$ is an entire function. And moreover, we have
	$$\lim_{q\to\infty}\int_{0}^{\infty}\frac{e^{\lambda'_i(t_q)}x^i}{\sum e^{\lambda'_j(t_q)}x^j}dx=\int_{0}^{\infty}\frac{e^{\lambda_i(\infty)}x^i}{\sum e^{\lambda_j(\infty)}x^j}dx. $$
	Therefore $$\int_{0}^{\infty}\frac{e^{\lambda_i(\infty)}x^i}{\sum_{j=0}^{\infty} e^{\lambda_j(\infty)}x^j}dx=1-\beta\delta_{0i}.$$
	It follows that $\blambda(\infty)\equiv\{\lambda_i(\infty)\}$ is a solution to the balanced equation \eqref{e-1}. This proves the existence part of Theorem \ref{t:main} under the assumption that $\beta\leq\frac12$.
	
	\
	
	Denote by $f(x)=\sum_{i=0}^{\infty}e^{\hat{\lambda}_i}x^i$ a solution of equation \eqref{e-evo} with $\beta_1=\frac12-\delta$ for some $\delta>0$ small. We then solve the differential equation \eqref{e-evo} again with initial value replaced by this new $\hat{\blambda}=(\hat{\lambda}_i)$. To see that we have the same conclusion as when the initial value is $\frac{1}{\Gamma(i+1)}$, one simply notices that the only place where we have used the properties of the initial value $\{\bar\lambda_i\}$ is in the proof of Lemma \ref{lem-e}. But since the coefficients of $f(x)$ satisfies the estimates \eqref{e-11} and \eqref{e-13}, $f(x)$ shares the same properties, namely
	\begin{itemize}
		\item Fix $x\geq 0$, then for for all $\epsilon>0$, There exists $ N>0$ such that $\sum_{i=N}^{\infty}e^{\hat{\lambda}_i}x^i<\epsilon$.
		\item Fix $i\geq 0$, then for  all $  \epsilon>0$, There exists $ N>0$ such that $\int_{i=N}^{\infty}\frac{e^{\hat{\lambda}_i}x^i}{f(x)}<\epsilon$.
	\end{itemize} 
	Therefore, we get solution $\hat\blambda(t)$ on $[0,\infty)$ for $\beta\geq \frac12-\delta$. 
	
	Now we have energy $E_0(\hat\blambda(t))\leq (\beta-\beta_1)^2 $, so $|F_0(\hat\blambda(t))|<\beta-(\frac12-\delta)$. So when $\beta\leq \frac{1+\beta_1-\delta}{2}$, we can also get upper bound \eqref{e-13}, hence solving the balancing equation \eqref{e-1}. We can the repeat this process to get a sequence $\beta_n$ satisfying $$\beta_{n+1}=\frac{1+\beta_n-\delta}{2}.$$
	Then since $\delta$ can be chosen arbitrarily small, we have proved the existence of solution  to \eqref{e-1} for all $\beta<1$.

	\bibliographystyle{plain}

	\bibliography{references}
\end{document}